\documentclass[11pt]{article}
\usepackage{geometry}
\usepackage{amssymb, amsthm, amsmath}
\usepackage{enumitem}
\usepackage{hyperref}
\usepackage{amsrefs}

\newtheorem{theorem}{Theorem}[section]
\newtheorem{lemma}[theorem]{Lemma}
\newtheorem{proposition}[theorem]{Proposition}

\theoremstyle{definition}
\newtheorem{definition}[theorem]{Definition}

\theoremstyle{remark}
\newtheorem{remark}[theorem]{Remark}

\title{An induced subgraph of the Hamming graph \\with maximum degree 1}
\author{Vincent Tandya \thanks{Email address: \href{mailto:vincent.tandya@u.nus.edu}{vincent.tandya@u.nus.edu}.}}
\date{}

\begin{document}

\maketitle

\begin{abstract}
For every graph $G$, let $\alpha(G)$ denote its independence number. What is the minimum of the maximum degree of an induced subgraph of $G$ with $\alpha(G)+1$ vertices? We study this question for the $n$-dimensional Hamming graph over an alphabet of size $k$. In this paper, we give a construction to prove that the answer is $1$ for all $n$ and $k$ with $k \geq 3$. This is an improvement over an earlier work showing that the answer is at most $\lceil \sqrt{n} \, \rceil$.
\end{abstract}


\section{Introduction} 
For every graph $G$, denote $\alpha(G)$ as the independence number of $G$, and $f(G)$ as the minimum of the maximum degree of an induced subgraph of $G$ with $\alpha(G)+1$ vertices. The Hamming graph $H(n, k)$ is a graph whose vertex set consists of $k^n$ vertices labelled by vectors in $\{0, 1, 2, \dots, k-1\}^n$, where two vertices are adjacent if and only if their corresponding vectors differ in exactly one coordinate. 

A special case of the Hamming graphs is the $n$-dimensional hypercube graph $H(n, 2)$, also often denoted by $Q^n$. The graph $Q^n$ is of interest due to the relation between $f(Q^n)$ and the Sensitivity Conjecture in the field of theoretical computer science. Chung, F\"{u}redi, Graham, and Seymour \cite{Chung} provided a construction proof to show that $f(Q^n) \leq \lceil \sqrt{n} \, \rceil$, and Huang \cite{Huang} later proved that $f(Q^n) \geq \lceil \sqrt{n} \, \rceil$. Huang's result is significant since it is known to be equivalent to the Sensitivity Conjecture. For readers interested in the conjecture, we refer to a survey by Hatami, Kulkarni, and Pankratov \cite{HKP} regarding several problems related to the conjecture, and another survey by Karthikeyan, Sinha, and Patil \cite{KSP} regarding results relevant to Huang's proof.

More recently, Dong \cite{Dong} generalised the upper bound established by Chung et al.\ to other Hamming graphs. By generalising their construction, Dong proved that the inequality $f(H(n, k)) \leq \lceil \sqrt{n} \, \rceil$ also holds when $k \geq 3$. However, the work did not establish a lower bound for the quantity.

In this paper, we improve Dong's result by proving the following theorem.

\begin{theorem} \label{theorem:main}
For all positive integers $n$ and $k$ with $k \geq 3$, there exists an induced subgraph of $H(n, k)$ with $\alpha(H(n, k))+1$ vertices and maximum degree $1$.
\end{theorem}

As an immediate consequence, we have $f(H(n, k)) = 1$ for all $k \geq 3$. Combined with the results by Chung et al.\ and Huang, we have the following conclusion.

\begin{theorem} \label{theorem:combined}
For all positive integers $n$, we have 
\[f(H(n, k)) = \begin{cases} \lceil \sqrt{n} \, \rceil & \text{if } k = 2, \\ 1 & \text{if } k \geq 3.\end{cases}\]
\end{theorem}

\section{Notation and preliminaries} 
For every graph $G$, let $V(G)$ denote the vertex set of $G$. For every vertex $v \in V(H(n, k))$ and every integer $1 \leq i \leq n$, let $v(i)$ denote the $i$-th coordinate of $v$. We can therefore view the vertex $v$ as the vector of $n$ coordinates $(v(1), v(2), \dots, v(n))$, where each coordinate is in $\mathbb{Z}_k := \{0, 1, 2, \dots, k-1\}$.

We will perform several arithmetic operations on the vertex coordinates. These are all performed in modulo $k$. For brevity, we write $a \equiv b$ to state that $a$ and $b$ are congruent modulo $k$, and we write $a \not \equiv b$ to state the opposite.

We now prove the following facts to determine the value of $\alpha(H(n, k))$.

\begin{definition} \label{definition:set_x}
For every $s \in \mathbb{Z}_k$, define the vertex set $X(s)$ as the set \[X(s) := \left\{v \in V(H(n, k)) \colon \sum_{i=1}^n v(i) \equiv s\right\}.\]
\end{definition}

\begin{lemma} \label{lemma:set_x_independent}
For every $s \in \mathbb{Z}_k$, the vertex set $X(s)$ is an independent vertex set.
\end{lemma}

\begin{proof}
It suffices to show that for every pair of adjacent vertices $v$ and $w$ in $H(n, k)$, there is no $s \in \mathbb{Z}_k$ such that both $v$ and $w$ are in $X(s)$. Indeed, if $v$ and $w$ are adjacent vertices, then $v$ and $w$ are different in exactly one coordinate, and thus \[\sum_{i=1}^n v(i) \not \equiv \sum_{i=1}^n w(i).\] As a result, we cannot have both $v \in X(s)$ and $w \in X(s)$ hold simultaneously.
\end{proof}

\begin{proposition} \label{proposition:independence_number}
The independence number of the Hamming graph $H(n, k)$ is $k^{n-1}$.
\end{proposition}

\begin{proof}
First, we show that $\alpha(H(n, k)) \leq k^{n-1}$ by contradiction. Let us assume that there exists an independent induced subgraph of $H(n, k)$ with $k^{n-1}+1$ vertices. By pigeonhole principle, there exist two distinct vertices in the subgraph whose first $n-1$ coordinates are identical. But this implies that the vertices are different only at the $n$-th coordinate. Hence, the two vertices are adjacent. This contradicts our initial assumption.

Next, we prove that the independent vertex set $X(s)$ defined in Definition~\ref{definition:set_x} has exactly $k^{n-1}$ vertices for every $s \in \mathbb{Z}_k$. Indeed, we know $v \in X(s)$ if and only if \[v(1) \equiv s - \sum_{i=2}^{n} v(i).\] Observe that for each of the possible $k^{n-1}$ values for $v(2), v(3), \dots, v(n)$, there is exactly one possible value for $v(1)$ to satisfy the equation. It follows that $|X(s)| = k^{n-1}$.
\end{proof}

\section{Construction of the induced subgraph}

As we have just shown, the vertex set $X(s)$ is independent and has exactly $\alpha(H(n, k))$ vertices for every $s \in \mathbb{Z}_k$. However, to prove Theorem~\ref{theorem:main}, we could not simply add another vertex to $X(s)$ since it would cause the maximum degree to suddenly jump to $n$. In fact, the added vertex would be adjacent to all vertices in $X(s)$. On the other hand, Lemma~\ref{lemma:set_x_independent} shows a useful property of $X(s)$. We are using this property to prove our main theorem.

We introduce the following definitions first.

\begin{definition} \label{definition:ell}
For every vertex $v \in V(H(n, k)) - \{(0, 0, \dots, 0)\}$, define $\ell(v)$ to be the largest integer $1 \leq i \leq n$ such that $v(i) \neq 0$.
\end{definition}

\begin{definition} \label{definition:set_y}
For every $s, t \in \mathbb{Z}_k$ with $t \neq 0$, define the vertex set $Y(s, t)$ as the set
\[Y(s, t) := \left\{v \in V(H(n, k)) - \{(0, 0, \dots, 0)\} \colon \sum_{i=1}^n v(i) \equiv s \text{ and } v(\ell(v)) = t\right\}.\]
In other words, the vertex set $Y(s, t)$ is the set of all vertices in $X(s)$ whose last nonzero coordinate is $t$.
\end{definition}

Next, we define the following vertex set to be used to prove Theorem~\ref{theorem:main}.

\begin{definition} \label{definition:set_w}
Define the vertex set $W$ as the set
\[W := Y(1, 1) \cup \left(\bigcup_{i=2}^{k-1} Y(2, i) \right).\] 
\end{definition}

\section{Analysis of the construction}
Recall that Proposition~\ref{proposition:independence_number} shows $\alpha(H(n, k)) = k^{n-1}$. In this section, we aim to prove Theorem~\ref{theorem:main} by proving that the subgraph of $H(n, k)$ induced by the vertex set $W$ defined in Definition~\ref{definition:set_w} has exactly $\alpha(H(n, k))+1 = k^{n-1}+1$ vertices and maximum degree $1$.

\subsection{Number of the vertices in the subgraph}

To find the number of vertices in $W$, it is sufficient to find $|Y(s, t)|$ for each $s, t \in \mathbb{Z}_k$.

\begin{lemma} \label{lemma:size_y}
For every $s, t \in \mathbb{Z}_k$ with $t \neq 0$, we have
\[
|Y(s, t)| = 
	\begin{cases} 
		(k^{n-1}-1)/(k-1) + 1 & \text{if } s = t, \\ 
		(k^{n-1}-1)/(k-1) & \text{otherwise}. 
	\end{cases}
\]
\end{lemma}

\begin{proof}
For every integer $1 \leq c \leq n$ and $s, t \in \mathbb{Z}_k$ with $t \neq 0$, define the vertex set $Z(s, t, c)$ as the set
\[Z(s, t, c) := \{v \in Y(s, t) \colon \ell(v) = c\}.\]
Then $v \in Z(s, t, c)$ if and only if $v$ satisfies the system
\begin{align*}
	\sum_{i=1}^n v(i) &\equiv s,\\
	v(c) &= t,\\
	v(j) &= 0 \quad \text{for all } c < j \leq n.
\end{align*}
Consider the following two cases.
\begin{itemize}
\item Suppose $c = 1$. Then the system is satisfied if and only if both $v = (s, 0, 0, \dots, 0)$ and $v = (t, 0, 0, \dots, 0)$ hold simultaneously. Consequently, 
\[
|Z(s, t, 1)| = 
	\begin{cases} 
		1 & \text{if } s = t, \\ 
		0 & \text{otherwise}.
	\end{cases}
\]
\item Suppose $2 \leq c \leq n$. Then, to satisfy the system, there is only one possible value for each of the coordinates $v(c), v(c+1), v(c+2), \dots, v(n)$; that is, $v(c) = t$ and $v(c+1) = v(c+2) = \dots = v(n) = 0$. 
We also have \[v(1) \equiv s - \sum_{i=2}^n v(i) \equiv  s - \left(\sum_{i=2}^{c-1} v(i)\right) - t.\]
Hence, for each of the possible $k^{c-2}$ values for $v(2), v(3), \dots, v(c-1)$, there is exactly one possible value for $v(1)$ to satisfy the system. It follows that $|Z(s, t, c)| = k^{c-2}$ for all $2 \leq c \leq n$. 
\end{itemize}
In conclusion,
\[
|Y(s, t)| 
	= \sum_{c=1}^n |Z(s, t, c)| 
	= 
	\begin{cases} 
		1 + k^0 + k^1 + k^2 + \cdots + k^{n-2} & \text{if } s = t, \\ 
		0 + k^0 + k^1 + k^2 + \cdots + k^{n-2} & \text{otherwise}. 
	\end{cases}
\]
A straightforward algebraic manipulation of the above gives the required result.
\end{proof}

Now, we are able to find the number of vertices in $W$.

\begin{lemma} \label{lemma:size_w}
There are exactly $k^{n-1}+1$ vertices in the vertex set $W$.
\end{lemma}

\begin{proof}
By Lemma~\ref{lemma:size_y}, we have
\begin{align*}
|W| 
	&= |Y(1, 1)| + |Y(2, 2)| + \sum_{i=3}^{k-1} |Y(2, i)| \\
	&= 2 \cdot \left(\frac{k^{n-1}-1}{k-1} + 1\right) + (k-3) \cdot \frac{k^{n-1}-1}{k-1} \\
	&= k^{n-1} + 1
\end{align*}
as desired.
\end{proof}

\begin{remark}
Note that our proof of Lemma \ref{lemma:size_w} does not hold when $k = 2$, because it involves $Y(2, 2)$ even though $2 \not \in \mathbb{Z}_2$. If we let $W = Y(1, 1)$ instead, then we would have $|W| = 2^{n-1} < \alpha(H(n, 2)) + 1$.
\end{remark}

\subsection{Vertex adjacency in the Hamming graph}
Next, we show that the maximum degree of the subgraph induced by $W$ is $1$. To find the maximum degree of the subgraph, it is natural to find some properties regarding adjacent and non-adjacent vertex pairs in the Hamming graph.

As mentioned earlier, we aim to exploit the property proven in Lemma~\ref{lemma:set_x_independent}. We are using that lemma now to prove the following fact.

\begin{lemma} \label{lemma:set_y_independent}
Let $s \in \mathbb{Z}_k$ and $T \subseteq \mathbb{Z}_k - \{0\}$. Then $\bigcup_{t \in T} Y(s, t)$ is independent.
\end{lemma}

\begin{proof}
The vertex set $\bigcup_{t \in T} Y(s, t)$ is a subset of the vertex set $X(s)$. We know $X(s)$ is independent by Lemma~\ref{lemma:set_x_independent}, so $\bigcup_{t \in T} Y(s, t)$ must be independent too.
\end{proof}

Of course, for some $s_1, s_2, t_1, t_2 \in \mathbb{Z}_k$, a vertex in $Y(s_1, t_1)$ and another vertex in $Y(s_2, t_2)$ may be adjacent. In that case, we have the following lemma.

\begin{lemma} \label{lemma:adjacent_vertices}
Let $s_1, s_2, t_1, t_2 \in \mathbb{Z}_k$ with $t_1 \neq 0$, $t_2 \neq 0$, and $t_1 \neq t_2$. Suppose there exists two vertices $v \in Y(s_1, t_1)$ and $w \in Y(s_2, t_2)$ that are adjacent. Then we have the following:
\begin{enumerate}[label=(\roman*)]
\item If $\ell(v) > \ell(w)$, then $s_1 - s_2 \equiv t_1$.
\item If $\ell(v) < \ell(w)$, then $s_2 - s_1 \equiv t_2$.
\item If $\ell(v) = \ell(w)$, then $s_1 - s_2 \equiv t_1 - t_2$.
\end{enumerate}
\end{lemma}

\begin{proof}
Suppose $\ell(v) > \ell(w)$. Then we have $v(\ell(v)) = t_1$ and $w(\ell(v)) = 0$. Consequently, the two vertices differ only at the $\ell(v)$-th coordinate, and thus we have
\[s_1 - s_2 \equiv \sum_{i=1}^n \Big(v(i) - w(i)\Big) \equiv v(\ell(v)) - w(\ell(v)) \equiv t_1 - 0\]
as required. We can similarly prove that $\ell(v) < \ell(w)$ implies $s_2 - s_1 \equiv t_2$.

Next, suppose $\ell(v) = \ell(w)$. Then \[v(\ell(v)) = t_1 \neq t_2 = w(\ell(w)) = w(\ell(v)).\] Hence, the two vertices differ only at the $\ell(v)$-th coordinate, and thus we have \[s_1 - s_2 \equiv \sum_{i=1}^n \Big(v(i) - w(i)\Big) \equiv v(\ell(v)) - w(\ell(v)) \equiv t_1 - t_2\]
as required. 
\end{proof}

As a consequence of Lemma~\ref{lemma:adjacent_vertices}, we have a sufficient condition for two vertices in $H(n, k)$ to be \emph{not} adjacent.

\begin{lemma} \label{lemma:independent_y}
Let $s_1, s_2, t_1, t_2 \in \mathbb{Z}_k$ with $t_1 \neq 0$, $t_2 \neq 0$, and $t_1 \neq t_2$. Suppose that the following statements are all false:
\begin{enumerate}[label=(\roman*)]
\item $s_1 - s_2 \equiv t_1$.
\item $s_2 - s_1 \equiv t_2$.
\item $s_1 - s_2 \equiv t_1 - t_2$.
\end{enumerate}
Then every vertex in $Y(s_1, t_1)$ is adjacent to none of the vertices in $Y(s_2, t_2)$.
\end{lemma}

\begin{proof}
By Lemma~\ref{lemma:adjacent_vertices}, a vertex in $Y(s_1, t_1)$ is adjacent to a vertex in $Y(s_2, t_2)$ only if at least one of the three statements above is true. Since we assume that all three statements are false, the two vertices must be not adjacent.
\end{proof}

Unfortunately, Lemma~\ref{lemma:independent_y} cannot be applied to all pairs of vertices in $W$. In particular, the lemma does not say anything about the relationship between vertices in $Y(1, 1)$ and vertices in $Y(2, 2)$. For that, we have the following lemma.

\begin{lemma} \label{lemma:adjacent_y}
Let $s_1, s_2, t_1, t_2 \in \mathbb{Z}_k$ with $t_1 \neq 0$, $t_2 \neq 0$, and $t_1 \neq t_2$. If $s_1 - s_2 \equiv t_1 - t_2$, then every vertex in $Y(s_1, t_1)$ is adjacent to exactly one vertex in $Y(s_2, t_2)$.
\end{lemma}

\begin{proof}
Let $v$ be an arbitrary vertex in $Y(s_1, t_1)$, and suppose there exists a vertex $w$ in  $Y(s_2, t_2)$ that is adjacent to $v$.  From the three conditions $t_1 \neq 0$, $t_2 \neq 0$, and $s_1 - s_2 \equiv t_1 - t_2$, we know that $s_1 - s_2 \not \equiv t_1$ and $s_2 - s_1 \not \equiv t_2$. So, by Lemma~\ref{lemma:adjacent_vertices}, we can deduce that $\ell(v) = \ell(w)$. It follows that \[v(\ell(v)) = t_1 \neq t_2 = w(\ell(w)) = w(\ell(v)),\] so the vertices differ at the $\ell(v)$-th coordinate. As a consequence, if $w$ exists, then $w$ must be the unique vertex obtained by replacing the $\ell(v)$-th coordinate of $v$ from $t_1$ to $t_2$.

It remains to verify that such $w$ is in $Y(s_2, t_2)$. Indeed, we know $w(\ell(w)) = t_2$ and
\[\sum_{i=1}^n w(i) \equiv \left(\sum_{i=1}^n v(i)\right) - v(\ell(v)) + w(\ell(v)) \equiv s_1 - t_1 + t_2 \equiv s_2,\]
so the vertex $w$ is in fact a vertex in $Y(s_2, t_2)$.
\end{proof}

We are finally ready to find the maximum degree of the induced subgraph.

\begin{lemma} \label{lemma:maximum_degree}
The maximum degree of the subgraph of $H(n, k)$ induced by $W$ is $1$.
\end{lemma}

\begin{proof}
By Lemma~\ref{lemma:set_y_independent}, the vertex sets $Y(1, 1)$ and $\bigcup_{i=2}^{k-1} Y(2, i)$ are both independent. It remains to check the adjacency between vertices in $Y(1, 1)$ and vertices in $Y(2, i)$ for each $2 \leq i \leq k-1$.
\begin{itemize}
\item Suppose $i = 2$. By Lemma~\ref{lemma:adjacent_y}, every vertex in $Y(1, 1)$ is adjacent to exactly one vertex in $Y(2, 2)$, and every vertex in $Y(2, 2)$ is adjacent to exactly one vertex in $Y(1, 1)$ too. 
\item Suppose $3 \leq i \leq k-1$. By Lemma~\ref{lemma:independent_y}, every vertex in $Y(1, 1)$ is adjacent to none of the vertices in $Y(2, i)$.
\end{itemize}
Therefore, the induced subgraph has maximum degree $1$ due to the vertices in $Y(1, 1)$ and $Y(2, 2)$.
\end{proof}

\begin{proof}[Proof of Theorem~\ref{theorem:main}]
Combine Lemma~\ref{lemma:size_w} and Lemma~\ref{lemma:maximum_degree}.
\end{proof}

\textbf{Acknowledgements.} This work is based on the author’s undergraduate final year project at National University of Singapore. The author would like to thank Professor Dilip Raghavan for introducing and discussing the Sensitivity Conjecture along with this Hamming graph problem.

\end{document}